\documentclass[a4paper, 11 pt, twoside]{amsart}
\usepackage{srollens-en}[2011/03/03]
\usepackage[left = 3.5cm, right = 3.5cm, headsep = 6mm,
footskip = 10mm, top = 35mm, bottom = 35mm, footnotesep=5mm, headheight =
2cm]{geometry}
\usepackage{amsmath,amssymb,amscd,amsfonts,verbatim}

\usepackage[usenames, dvipsnames]{xcolor}

\usepackage[T1]{fontenc}

\usepackage{tikz, tikz-cd}
\usetikzlibrary{positioning, intersections}
\usetikzlibrary{fit}
\tikzset{corners/.style={fit={#1},rectangle,inner sep=0, outer sep = .2cm}}

\usepackage{booktabs, multirow}
\usepackage{enumitem}

\usepackage[unicode,bookmarks, pdftex, backref = page]{hyperref}
\hypersetup{colorlinks=true,citecolor=NavyBlue,linkcolor=NavyBlue,urlcolor=Orange, pdfpagemode=UseNone, breaklinks=true}

\renewcommand{\tilde}{\widetilde}

\newcommand{\pp}{\mathbb P}

\newcommand{\OO}{\mathcal O}
\newcommand{\Diff}{\mathrm{Diff}}

\newcommand{\wt}{\widetilde}

\newcommand{\epsi}{\epsilon}
\newcommand{\inv}{^{-1}}

\newcommand{\I}{\mathrm{i}}

\title{Log-canonical pairs and Gorenstein stable surfaces with $K_X^2=1$}

\author{Marco Franciosi}
\address{Marco Franciosi\\Dipartimento di Matematica\\Universit\`a di Pisa \\Largo B. Pontecorvo 5\\I-56127  Pisa\\Italy}
\email{franciosi@dm.unipi.it}

\author{Rita Pardini}
\address{Rita Pardini\\Dipartimento di Matematica\\Universit\`a di Pisa \\Largo B. Pontecorvo 5\\I-56127  Pisa\\Italy}
\email{pardini@dm.unipi.it}

\author{S\"onke Rollenske}
\address{S\"onke Rollenske\\Fakult\"at f\"ur Mathematik\\Universit\"at Bielefeld\\Universit\"atsstr. 25\\33615 Bielefeld\\Germany}
\email{rollenske@math.uni-bielefeld.de}

\begin{document}
\begin{abstract}
We classify log-canonical pairs $(X, \Delta)$ of dimension two with $K_X+\Delta$ an ample Cartier divisor with $(K_X+\Delta)^2=1$, giving  some  applications to  stable surfaces with $K^2=1$.
 A rough classification is also given in the case  $\Delta=0$. 
\end{abstract}


\maketitle

\setcounter{tocdepth}{1}
\tableofcontents
\section{Introduction}

The study of stable curves and, more generally, stable pointed curves is by now a classical subject. 
Stable surfaces were introduced by Koll\'ar and Shepherd-Barron in \cite{ksb88} and it was consequently realized (see, for instance,  \cite{alexeev06, kollar12,  KollarModuli}  and references therein)
that
this definition can be extended to higher-dimensional varieties and
pairs.
So the study of (semi-)log-canonical  pairs  became  an important topic in the  theory of  singular  higher-dimensional varieties.

Here  we consider two-dimensional log-canonical pairs    in which  the log-canonical divisor is Cartier and has self-intersection equal to 1, and we give
 some applications to Gorenstein stable surfaces. 
 
First we study  the case with non-empty boundary:
\begin{thm}\label{thm: pairs}  Let $(X, \Delta)$ be a    log-canonical  pair   of dimension 2 with   $\Delta>0$, $K_X+\Delta$ Cartier and ample  and $(K_X+\Delta)^2=1$.

Then $(X,\Delta)$ belongs to one of the types $(P)$, $(dP)$, $(E_+)$ or $(E_-)$ described in List~\ref{list}.
\end{thm}

In particular,   Theorem \ref{thm: pairs} implies that $X$ is either the projective plane, a del Pezzo surface of degree 1, the symmetric product $S^2E$ of an elliptic curve, or a projective bundle $\IP(\ko_E\oplus\ko_E(x))$ over an elliptic curve with the section of square $-1$ contracted.
It came rather as a suprise to us that the list is so short and that in each case the underlying surface itself is Gorenstein.

The case in which $\Delta=0$ cannot be described so precisely, since it includes, for instance, all smooth surfaces of general type with $K^2=1$; however in  Section \ref{section: bigandnef} we  give a rough classification, according to the Kodaira dimension of a smooth model of $X$ (see Theorem \ref{thm:normal-case}). 

Although log-canonical pairs are interesting in their own right, our  main motivation for proving  the above results is that,  by a result of Koll\'ar,  a non-normal Gorenstein stable surface gives rise to a pair as in  Theorem \ref{thm: pairs} via normalisation (see Corollary \ref{cor: main motivation}).  In Section \ref{section: applications to moduli}, we explain how the above pairs can be used to construct stable surfaces and which pairs can occur as normalisations of stable surfaces  for given invariants $K^2 $ and $\chi$.
 In particular, we show that $\chi(X)\geq0$ for a Gorenstein stable surface $X$  with $K_X^2=1$ improving upon results in \cite{liu-rollenske13}.

 We will study the geometry and moduli of Gorenstein stable surfaces with $K^2=1$ more in detail in a subsequent paper, building on the classification results proven here.

\subsection*{Acknowledgements}
We are grateful to 
 Wenfei Liu
for many discussions on stable surfaces and related topics and to Valery Alexeev for some useful communications. 

 The first author  is a member of GNSAGA of INDAM. The third author is grateful for support of the DFG through the Emmy-Noether program and SFB 701; he enjoyed the hospitality of HIM in Bonn during the final preparation of this paper. The collaboration  benefited immensely from a visit of the third author in Pisa supported by GNSAGA of INDAM. This project was partially supported by PRIN 2010 ``Geometria delle Variet\`a Algebriche'' of italian MIUR. 
 
\subsection*{Notation and conventions. } We work over the complex numbers; all varieties are assumed to be projective and irreducible unless otherwise stated. We do not distinguish between Cartier divisors and invertible sheaves in our notation. For a variety $X$ we denote by $\chi(X)$ the holomorphic Euler-characteristic and by $K_X$ a canonical divisor.

\section{Classification of pairs}\label{sec:pairs}
 Let $(X, \Delta)$ be a log-canonical (lc) pair of dimension two (cf.  \cite[Def. 2.34]{Kollar-Mori} for the definition). 
\begin{defin}
   We call $(X, \Delta) $ stable if $K_X+\Delta$ is ample and Gorenstein if $K_X+\Delta$ is Cartier.
\end{defin}

The aim of this section is the classification of  Gorenstein stable lc pairs  with $(K_X+\Delta)^2=1$ and $\Delta> 0 $. We start by listing and  describing quickly the cases that occur in our classification.

\begin{List}\label{list}~
\begin{itemize}
 \item[$(P)$] $X=\pp^2$ and $\Delta $ is a  nodal quartic. Here $p_a(\Delta)=3$ and  $K_X+\Delta=\OO_{\pp^2}(1)$.  
 
 \item[$(dP)$] $X$ is a (possibly singular) Del Pezzo surface of degree $1$, namely $X$ has at most canonical singularities,  $-K_X$ is ample and $K^2_X=1$.  The curve  $\Delta$ belongs to the system $|-2K_{X}|$, hence  $K_X+\Delta=-K_X$ and $p_a(\Delta)=2$. 
 \item[$(E_-)$]  Let $E$ be an elliptic curve and let  $a\colon \wt X\to E$ be a geometrically ruled surface that contains an irreducible section  $C_0$ with $C_0^2=-1$. Namely, $\tilde X=\pp(\OO_E\oplus \OO_E(-x))$, where $x\in E$ is a point  and $C_0$ is the only curve in the system $|\OO_X(1)|$.  Set $F=a\inv(x)$: the normal surface   $X$ is obtained  from $\wt X$ by contracting $C_0$ to an elliptic Gorenstein singularity of degree 1 and $\Delta$ is the image of a curve $\Delta_0\in |2(C_0+F)|$ disjoint from $C_0$, so $p_a(\Delta)=2$. The line bundle $K_X+\Delta$ pulls back  to $C_0+F$ on $\wt X$. 
  \item[$(E_+)$]
    $X=S^2E$, where $E$ is an elliptic curve. Let $a\colon X\to E$ be the Albanese map, which is induced by the addition map $E\times E \to E$, denote by $F$ the class of a fiber of $a$ and by $C_0$ the image in $X$ of the  curve  $\{0\}\times E+E\times \{0\}$, where $0\in E$ is the origin, so that $C_0 F=C_0^2=1$. Then $\Delta$ is a divisor numerically equivalent to $3C_0-F$, $p_a(\Delta)=2$  and $K_X+\Delta$ is numerically equivalent to $C_0$. 
            
An equivalent description of $X$ is as follows (cf. \cite[\S1]{catanese-ciliberto93}). Denote by $\ke$ the only indecomposable extension of the form  $0\to \OO_E\to \ke\to \OO_E(0)\to 0$ and  set  $X=\pp(\ke)$: then  $C_0$ is the only effective divisor in $|\OO_X(1)|$.
\end{itemize}
\end{List}
For completeness,  we give in Table \ref{tab: invariants pairs} the numerical  invariants of the four  possible cases.
\begin{table}[htb!]\caption{Invariants of $( X, \Delta)$}\label{tab: invariants pairs}
 \begin{center}
 \begin{tabular}{c ccc c}
 \toprule
Case & $\chi({X})$ & $q(X)$ & $p_a(\Delta)$ & $h^0(K_{X}+\Delta)$ \\
\midrule
$(P)$ 	& 1 	& 0	& 3	& 3\\
$(dP)$ 	& 1 	& 0	& 2	& 2\\
$(E_-)$	& 1 	& 0	& 2	& 2\\
$(E_+)$	& 0 	& 1	& 2	& 1\\
\bottomrule
\end{tabular}
\end{center}
\end{table}
The rest of the section is devoted to proving Theorem \ref{thm: pairs}. We start with  some general   remarks: 
\begin{lem}\label{lem: h^0>=3} Let $X$ be a normal surface and let $L$ be an ample line bundle of $X$ such that $L^2=1$. Then:
\begin{enumerate}
\item every curve $C\in|L|$ is irreducible and $h^0(L)\le 3$
\item   $h^0(L)=3$ if and only if $X=\pp^2$ and $L=\OO_{\pp^2}(1)$
\item if $h^0(L)=2$, then the system $|L|$ has one simple base point $P$ that is smooth for $X$. 
\end{enumerate}
\end{lem}
\begin{proof}  \refenum{i}, \refenum{ii}
We have $LC=1$, hence $C$ is irreducible, since $L$ is ample. 
Denote by $\nu\colon \wt C \to C$ the normalization: since $\deg L|_C=1$, one has $h^0(\nu^*L)\le 2$, with equality holding iff $\wt C$ is a smooth rational curve. Since $h^0(L|_C)\le h^0(\nu^*L)$, the usual  restriction sequence 
$$0\to \OO_X\to \OO_X(C)=L\to L|_C\to 0$$
 gives $h^0(L)\le 3$. Moreover, if $h^0(L)=3$ then $h^0(L|_C)=h^0(\nu^*L)=2$,  $C$ is a smooth rational curve and the system $|L|$ is base point free. The   morphism $X\to \pp^2$ defined by $|L|$ has degree 1 and is finite, since $L$ is ample, so it is an isomorphism.

\refenum{iii} Follows by \refenum{i} and by the fact that $L^2=1$.
\end{proof}

\begin{lem}\label{lem:Dred}
Let $Y$ be a smooth surface, let $D>0$ be a nef and big divisor of $Y$ and let $D\red$ be the underlying reduced divisor.
Then:
\begin{enumerate}
\item $p_a(D\red)\le p_a(D)$
\item the natural  map $\Pic^0(Y)\to \Pic^0(D\red)$ is injective.
\end{enumerate}
\end{lem}
\begin{proof}
\refenum{i}
One has $h^1(K_Y+D)=0$ by Kawamata-Viehweg's vanishing, thus taking cohomology in the usual restriction sequence $0\to K_Y\to K_Y+D\to K_D\to 0$ one obtains 
\[p_a(D)=\chi(K_D)+1=\chi(K_Y+D)-\chi(K_Y)+1=h^0(K_Y+D)-\chi(K_Y)+1\]
Applying the same argument   to $D\red$ one obtains instead the  inequality:
\[ p_a(D\red)\le h^0(K_Y+D\red)-\chi(K_Y)+1,\]
since $h^2(K_Y+D\red)=h^0(-D\red)=0$.
Then  the claim follows since $h^0(K_Y+D\red)\le h^0(K_Y+D)$.
\smallskip

\refenum{ii} This is a slight generalization of  \cite[Prop. 1.6]{CFMM} and can be proven  exactly by the same argument.

\end{proof}
Next we fix the notation and  the assumptions that we keep throughout  the rest of the section: $(X,\Delta)$ is an  lc pair satisfying the assumptions of Theorem \ref{thm: pairs} and $\epsi\colon \wt X\to X$ is the minimal desingularization. We set $L:=K_X+\Delta$,  and $\wt L:=\epsi^*L$;  $\wt L$ is a nef and big divisor with $\wt L^2=1$ and $h^0(L)=h^0(\wt L)$. We define the divisor $\wt \Delta$ by the equality $\wt L=K_{\wt X}+\wt \Delta$ and by requiring that $\epsilon_*\tilde \Delta = \Delta$. 

\begin{lem}\label{lem:h0} In the above set-up:
\begin{enumerate}
\item $K_{\wt X}\wt L<0$,   $h^2(\wt L)=0$
\item $\wt X$ is ruled. 
\end{enumerate}
\end{lem}

\begin{proof}
\refenum{i} 
 Using    the projection formula, we compute
  \[\wt L \wt \Delta=\epsilon^* L(\inverse\epsilon)_* \Delta=L\Delta= (K_{ X}+ \Delta)\Delta,\] so $\tilde L\tilde \Delta$ is a positive number and it is even,  by adjunction. Thus
\[ \wt L K_{\wt X} =  \wt L^2 -\wt L\wt \Delta=1-\wt L \wt \Delta<0 .\]
By Serre duality, we have  $h^2(\wt L)=h^0(-\wt \Delta)=0$, since $\wt L\wt \Delta=L\Delta>0$ and $\wt L$ is nef. 
\smallskip

\refenum{ii} Since $\wt L$ is nef,  the condition  $K_{\wt X}\wt L<0$ implies that  $\kappa(\wt X)=-\infty$.
\end{proof}

Next  we look at  the adjoint divisor $K_{\wt X}+\wt L$:
\begin{lem}
\label{lem:RE} Assume that $h^0(\wt L)\le 2$; then $K_{\wt X}\wt L=-1$,  and  there are  the following two  possibilities:
 \begin{itemize}
\item[$(R)$] $h^0(K_{\wt X}+\wt L)=\chi(\wt X)=1$ and $h^0(\wt L)=2$,
\item[$(E)$]  $h^0(K_{\wt X}+\wt L)=\chi(\wt X)=0$ and $h^0(\wt L)=2\text{ or } 1$.
\end{itemize}
\end{lem}
\begin{proof} Since $\wt L$ is nef and big, Riemann-Roch and Kawamata-Viehweg  vanishing give:
 \begin{equation}\label{eq:K+L}
 h^0(K_{\wt X}+\wt  L)= \chi(\wt X)+\frac{\wt L^2+K_{\wt X}\wt L}{2}=\chi(\wt X)+\frac{1+K_{\wt X}\wt L}{2}\le \chi(\wt X),
 \end{equation}
 where the last inequality follows by Lemma \ref{lem:h0}. 
 Since $\wt X$ is ruled by Lemma \ref{lem:h0}, we have $\chi(\wt X)\le 1$, so $h^0(K_{\wt X}+\wt L)\le 1$ and if equality holds, then $\chi(\wt X)=1$ and $\wt LK_{\wt X}=-1$.
 
 Assume  $h^0(K_{\wt X}+\wt L)=0$. Then equation \eqref{eq:K+L} implies that
 either  $\chi(\wt X)=1$,  $K_{\wt X}\wt L=-3$ or $\chi(\wt X)=0$ and $K_{\wt X}\wt L=-1$. In the first case, using  Lemma \ref{lem:h0} and Riemann-Roch we obtain  $h^0(\wt L)\ge \chi(\wt L)= 3$, against  the assumptions. 
 In the second case, since $K_{\wt X}\wt L=-1$,  the same argument gives $h^0(\wt L)\ge \chi(\wt L)=\chi(\wt X)+1$ which gives the listed cases.
\end{proof}
Case $(R)$ of the above Lemma gives case $(dP)$ in our classification:

\begin{lem} \label{lem:dP}
If $(X, \Delta)$ is as in case $(R)$ of Lemma \ref{lem:RE} then it is of type $(dP)$.
 \end{lem} 
\begin{proof} 
By Lemma \ref{lem: h^0>=3}, the base locus of the  pencil $|\wt L|=\epsi^*|L|$ is a simple point $\wt P$ which is the preimage of a smooth point $P\in X$; by adjunction the general $C\in |\wt L|$ is a smooth elliptic curve. Blowing up the point $P$ we get an elliptic fibration $p\colon \widehat X\to \pp^1$ with a section $\Gamma$.
  
   Denote by $Z$ the only effective divisor in $|K_{\wt X}+\wt L|$. Since $\wt L Z=0$, $Z$ does not contain the point $\wt P$ and it is contained in a finite union of curves of $|\wt L|$, hence it can be identified with a divisor $Z'$  of $\widehat X$  that is contained in a union of fibers of $p$ and  does not intersect the section $\Gamma$. 
   By the Kodaira classification of elliptic fibers, $Z'$ is either $0$ or it  is supported on a  set $R_1,\dots R_k$ of $-2$-curves;   the same is true for $Z$, since $Z'$ does not meet $\Gamma$. In particular, we have $K_{\wt X}Z=0$, hence 
   $$Z^2=ZK_{\wt X}+Z\wt L=0,$$
   and therefore $Z=0$ by the Index Theorem. So $\wt L=-K_{\wt X}$,   $X$ is a the anti-canonical model of $\wt X$ and $\wt D\in |-2K_{\wt X}|$. 
\end{proof}

We now turn to studying  case $(E)$ of Lemma \ref{lem:RE}. 
This gives rise to the cases $(E_-)$ and $(E_+)$ in our classification, depending on the value of $h^0(\wt L)$. 
\begin{lem}\label{lem:proj}
If $(X, \Delta)$ is as in case $(E)$ of Lemma \ref{lem:RE}, then
there exists an elliptic curve $E$ and a vector bundle $\ke$ on $E$ of rank 2  and degree 1 such that $\wt X=\pp(\ke)$ and $\wt L=\OO_{\wt X}(1)$.
\end{lem}
\begin{proof} 
By  Lemma \ref{lem:h0} and Lemma \ref{lem:RE}, the surface $\wt X$ is ruled and $q(\wt X)=1$; we denote by $a\colon \wt X\to E$ the Albanese map and by $F$ a fiber of $a$. 
\smallskip

\noindent{\bf Step 1:} {\em one has $\wt L F=1$}\newline
The linear system $|\wt L|$ is non-empty by Lemma \ref{lem:RE}. Fix $C\in |\wt L|$ and denote by $C\red$ the underlying reduced divisor. One has $p_a(C)=1$ by adjunction  and $p_a(C\red)\le 1$ by Lemma \ref{lem:Dred}. The natural map   $\Pic^0(E)=\Pic^0(X)\to \Pic^0(C\red)$ is an inclusion by Lemma \ref{lem:Dred}. Thus $p_a(C\red)=1$ and $\Pic^0(E)\to \Pic^0(C\red)$ is  an isomorphism. By \cite[Ch. 9, Cor. 12]{BLRNeron},
 $C\red=C_0+Z$, where $C_0$ is an elliptic curve that is mapped isomorphically onto $E$ by $a$, $Z$ is a sum  of smooth rational curves and the dual graph of $C\red$ is a tree. 
We write $C=bC_0+Z'$, where $b>0$ is an integer and $Z'$ has the same support as $Z$. If $b=1$, then $\wt L F=1$ as claimed. 

So assume by contradiction that $b>1$: in this case $1=\wt L^2\ge b\wt L C_0$ gives $\wt LC_0=0$. Then $C_0^2<0$,  $C_0$ is contracted by $\wt L$ to an elliptic singularity and it does not intersect any other $\epsi$-exceptional curve. Since $\wt L$ is nef and $\wt L C=\wt L^2=1$, there is exactly one component $\Gamma$ of $C$ that has nonzero intersection with $\wt L$, and $\Gamma$ appears in $C$ with multiplicity 1. In particular, $Z'-\Gamma$ is contracted by $\epsilon$ and therefore $C_0(Z'-\Gamma)=0$.
We have $C_0\Gamma \le 1$, since $\Gamma$ is contained in a fiber of $a$. 
Hence we have 
\[0=C_0\wt L=C_0(bC_0+\Gamma+(Z'-\Gamma))=bC_0^2+C_0\Gamma\le 1-b<0,\]
a contradiction.
\smallskip

\noindent{\bf Step 2:} {\em conclusion of the proof}\newline
We claim   that $a\colon \wt X\to E$ is a $\pp^1$-bundle. Indeed, assume by contradiction  that $\wt X$ contains an irreducible $(-1)$-curve $\Gamma$: then $\wt L \Gamma>0$, because $\wt X\to X$ is the minimal resolution and $\wt L$ is the pull back of an ample line bundle on $X$. On the other hand   $\wt L\Gamma\le \wt L F=1$, since $\Gamma$ is contained in a fiber $F$ of $a$. Hence $\wt L\Gamma=1$.
 But then we have $\wt L(F-\Gamma)=0$ and $K_{\wt X}(F-\Gamma)=-1$, namely $F-\Gamma$ contains a $(-1)$-curve $\Gamma_1$ with $\wt L\Gamma_1=0$, a contradiction.

Finally, we set $\ke=a_*\wt L$.  
\end{proof}

\begin{lem}\label{lem:E-+}
 Assume we are in case $(E)$ of Lemma \ref{lem:RE}.
\begin{enumerate}
\item If $h^0(\wt L)=2$, then $(X,\Delta)$ is of type $(E_-)$.
\item If $h^0(\wt L)=1$, then $(X,\Delta)$ is of type $(E_+)$.
\end{enumerate} 
\end{lem} 
\begin{proof} By  Lemma \ref{lem:proj} there exists an elliptic curve $E$ and a vector bundle $\ke$ on $E$ of rank 2 and degree 1 such that $\wt X=\pp(\ke)$ and $\wt L=\OO_{\wt X}(1)$. 
Denote by $x\in E$ the point such that $\det \ke=\OO_E(x)$. 
We will freely use the general theory of $\IP^1$-bundles and especially the classification of such bundles over an elliptic curve, see \cite[Ch.~V.2]{Hartshorne}.

Assume that $\ke$ is decomposable, i.e., that there are line bundles $A$ and $B$ on $E$ such that $\ke=A\oplus B$. Then we have $\deg A+\deg B=\deg \ke =1$ and $1\le h^0(A)+h^0(B)=h^0(\wt L)\le 2$.  So  there are three possibilities:
\begin{itemize}
\item[(a)] $\deg A=-1$, $\deg B=2$;
\item[(b)] $\deg A=0$, $A\ne \OO_E$ and   $\deg B=1$;
\item[(c)] $A=\OO_E$ and $B=\OO_E(x)$.
\end{itemize}
We denote by $C_0$ the section of $\wt X$ corresponding to the surjection $\ke \twoheadrightarrow A$. 
In case (a), the system $|\wt L|=|\OO_{\wt X}(1)|$ has dimension 1 and has $C_0$ as   fixed part, contradicting Lemma \ref{lem: h^0>=3}. So this case does not occur. 
In case (b), we have $\wt L C_0=0$, but $\wt L|_{C_0}$ is non-trivial: this contradicts the assumption that $\wt L$ is the pull-back of an ample line bundle via the birational map $\epsi\colon \wt X\to X$.  So (c) is the only possibility.   
In this case $C_0$ is contracted  to an elliptic singularity of degree 1 by $\epsi$ and $C_0$ is the only curve contracted by $\epsi$ since
$\NS(\wt X)$ has rank 2.
We have   $\wt \Delta=\wt L-K_{\wt X}=3C_0+2F$.  Since $K_{\wt X} =\epsi^*K_X-C_0$ and  $\Delta$ does not go through the elliptic singularity 
of $X$ because  the pair $(X,\Delta)$ is lc, we obtain $\epsi^*\Delta=\wt \Delta- C_0=2C_0+2F$ and $(X,\Delta)$ is a log surface of type $(E_-)$.

If $\ke$ is indecomposable, then  $\ke$ is the only non-trivial extension $0\to\OO_E\to\ke\to\OO_E(x)\to 0$ and $h^0(\wt L)=h^0(\ke)=1$.  Up to a translation in $E$, we may assume that $x$ is the origin $0\in E$. Hence  $\wt X=S^2E$  and $C=C_0=\wt L$ is the image of the curve $\{0\}\times E+E\times \{0\}$ via the quotient map $E\times E\to S^2E$ (cf. description of case $(E_+)$ at the beginning of the section).
 Since $\wt L$ is ample, we have $\wt X=X$, $\wt L= L$ and $\wt \Delta=\Delta=L-K_X$ is  numerically equivalent to $3C_0-F$. So the pair $(X,\Delta)$ is of type $(E_+)$.
 \end{proof}

 Finally,  we summarize all the above results:
\begin{proof}[Proof of Theorem \ref{thm: pairs}]
If $h^0(\wt L)\ge 3$, then by Lemma \ref{lem: h^0>=3} we have $X=\pp^2$ and $L=\OO_{\pp^2}(1)$, and  thus $(X,\Delta)$ is of type $(P)$.

So we may assume $h^0(\wt L)\le 2$, which by Lemma \ref{lem:RE} leaves us with the cases $(R)$ and $(E)$, according to the value of  $\chi({\tilde X})$. The first case gives type $(dP)$ by Lemma \ref{lem:dP} while the second splits up into the cases $(E_+)$ and $(E_-)$ by Lemma \ref{lem:E-+}. This concludes the proof of the Theorem.
\end{proof}

\section{Applications to stable surfaces}\label{section: applications to moduli}
In this section we explore some  consequences of the classification of pairs in Theorem \ref{thm: pairs} for the study of  stable surfaces with $K^2=1$. 
\subsection{Definitions and  Koll\'ar's gluing construction}\label{section: definitions}
 Our main reference for this section  is \cite[Sect.~5.1--5.3]{KollarSMMP}.

\subsubsection{Stable surfaces}
Let $X$ be a demi-normal surface, that is,  $X$ satisfies $S_2$ and  at each point of codimension one $X$ is either regular or has an ordinary double point.
We denote by  $\pi\colon \bar X \to X$ the normalisation of $X$.
Contrary to our previous assumptions $X$ is not assumed irreducible, in particular, $\bar X$ is possibly disconnected.
The conductor ideal
$ \shom_{\ko_X}(\pi_*\ko_{\bar X}, \ko_X)$
is an ideal sheaf in both $\ko_X$ and $\ko_{\bar X} $ and as such defines subschemes
$D\subset X \text{ and } \bar D\subset \bar X,$
both reduced and pure of codimension 1; we often refer to $D$ as the non-normal locus of $X$.

\begin{defin}\label{defin: slc}
The  demi-normal surface $X$ is said to have \emph{semi-log-canonical (slc)}  singularities if it satisfies the following conditions: 
\begin{enumerate}
 \item The canonical divisor $K_X$ is $\IQ$-Cartier.
\item The pair $(\bar X, \bar D)$ has log-canonical (lc) singularities. 
\end{enumerate}
It  is called a stable  surface 
 if in addition $K_X$ is ample. In that case we define the geometric genus of $X$ to be $ p_g(X) = h^0(X, K_X) = h^2(X, \ko_X)$ and the irregularity as $q(X) = h^1(X, K_X) = h^1(X, \OO_X)$. 
A Gorenstein stable surface is a stable surface such that $K_X$ is a Cartier divisor.
\end{defin}
The importance of these surfaces lies in the fact that they generalise stable curves: there is a projective moduli space of stable surfaces which compactifies the Gieseker moduli space of canonical model of surfaces of general type \cite{KollarModuli}.

\subsubsection{Koll\'ar's gluing principle}\label{ssec:kollar-glue}
Let $X$ be a demi-normal surface as above. Since $X$ has at most double points in codimension one, the map $\pi\colon \bar D \to D$ on the conductor divisors is generically a double cover and thus  induces a rational  involution on $\bar D$. Normalising the conductor loci we get an honest involution $\tau\colon \bar D^\nu\to \bar D^\nu$ such that $D^\nu = \bar D^\nu/\tau$ and such that $\Diff_{\bar D^\nu}(0)$ is $\tau$-invariant  (for the definition of the different see for example \cite[5.11]{KollarSMMP}).
\begin{theo}[{\cite[Thm.~5.13]{KollarSMMP}}]\label{thm: triple}
Associating to a stable surface $X$ the triple $(\bar X, \bar D, \tau\colon \bar D^\nu\to \bar D^\nu)$ induces a one-to-one correspondence
 \[
  \left\{ \text{\begin{minipage}{.12\textwidth}
 \begin{center}
         stable  surfaces  
 \end{center}
         \end{minipage}}
 \right\} \leftrightarrow
 \left\{ (\bar X, \bar D, \tau)\left|\,\text{\begin{minipage}{.37\textwidth}
   $(\bar X, \bar D)$ log-canonical pair with 
  $K_{\bar X}+\bar D$ ample, \\
   $\tau\colon \bar D^\nu\to \bar D^\nu$  involution s.th.\
    $\Diff_{\bar D^\nu}(0)$ is $\tau$-invariant.
            \end{minipage}}\right.
 \right\}.
 \]
\textbf{\upshape Addendum:} In the above correspondence the surface $X$ is Gorenstein if and only if  $K_{\bar X}+\bar D$ is Cartier and $\tau$ induces a fixed-point free involution on the preimages of the nodes of $\bar D$.
 \end{theo}

An important consequence, which allows to understand the geometry of stable  surfaces from the normalisation, is that \begin{equation}\label{diagr: pushout}
\begin{tikzcd}
    \bar X \dar{\pi}\rar[hookleftarrow]{\bar\iota} & \bar D\dar{\pi} & \bar D^\nu \lar[swap]{\bar\nu}\dar{/\tau}
    \\
X\rar[hookleftarrow]{\iota} &D &D^\nu\lar[swap]{\nu}
    \end{tikzcd}
\end{equation}
is a pushout diagram.

\begin{proof}[Proof of the Addendum in  Theorem \ref{thm: triple}]
Clearly, if $X$ is Gorenstein then $K_{\bar X}+\bar D=\pi^*K_X$ is an ample Cartier divisor. The converse follows from the classification of slc surface singularities in terms of the minimal semi-resolution \cite[Prop.~4.27]{ksb88}. More precisely, in the Gorenstein case the only singular points of $\bar X$ along $\bar D$ are contained in nodes of $\bar D$ and the different $\Diff_{\bar D^\nu}(0)$ is the sum of preimages of the nodes, each with coefficient 1. Thus the $\tau$-invariance of the different gives the action on the preimages of the nodes of $\bar D$.  Let $P\in X$ be the image of a node of $\bar D$. If $\tau$ fixes a point in the preimage of  $P$  in $\bar D^\nu$ then the  exceptional divisor over $P$  in the minimal semi-resolution cannot be a cycle of rational curves. Therefore, by classification the non-normal point $P$ is a quotient of a degenerate cusp  and it is not Gorenstein. This proves the remaining claim.
\end{proof}
Computing the main invariants of a stable surface from its normalisation is not difficult, see for example \cite[Prop.~2.5]{liu-rollenske13}.
\begin{prop}\label{prop: invariants}
 Let $X$ be a stable surface with normalisation $(\bar X,\bar D)$. Then
 $K_X^2 = (K_{\bar X}+\bar D)^2$  and  $\chi(X) = \chi({\bar X})+\chi(D)-\chi({\bar D})$.
\end{prop}
Note in particular that, by Nakai-Moishezon, a  Gorenstein stable surface with $K_X^2 = 1$ is irreducible.
Summing up,  we now state our main motivation for the classification in Theorem \ref{thm: pairs} explicitly:
\begin{cor}\label{cor: main motivation}
 Let  $X$ be a Gorenstein stable surface with $K_X^2 = 1 $ and let  $(\bar X, \bar D, \tau)$ be the corresponding triple as above. Then $(\bar X, \bar D)$ is of one of the types classified in Theorem \ref{thm: pairs}.
\end{cor}

\subsection{Numerology}\label{section: Numerology}
In this section we feed  the classification from Section \ref{sec:pairs} into Koll\'ar's gluing construction. The result is a precise list of the possible  normalisations  of a non-normal Gorenstein stable surface with $K_X^2 = 1$. We also give the possible values of $\chi(X)$ for each type, showing in particular  that there are no Gorenstein stable surfaces with $K_X^2 =1$ and $\chi(X)<0$.

We start with a preliminary lemma. In order to state it,  we keep the notation from Section \ref{ssec:kollar-glue} and introduce some additional numerical invariants of a stable surface  $X$:
\begin{itemize}
\item $\mu_1$, the number of degenerate cusps 
\item $\mu_2$,  the number of $\IZ/2\IZ$-quotients of degenerate cusps of $X$ 
\item $\rho$, the number of ramification points of the map $\bar D^\nu \to D^\nu$
\item $\bar \mu$ the number of nodes of $\bar D$
\end{itemize}
\begin{lem}\label{lem: chi(D)} Let $X$ be a non-normal stable surface. With the above notation: 
\begin{enumerate}
\item $\chi(D) = \frac 12\left( \chi({\bar D})-\bar \mu\right)+\frac\rho 4+\mu_1$.
\item If  $K_{\bar X}+\bar D$ is Cartier then $\chi(D) \geq 2\chi({\bar D})+\frac \rho 4 +\mu_1$. 
\item If $X$ is Gorenstein,  then $\chi(D)\geq 2\chi({\bar D})+1$.
 \end{enumerate}
 In addition, if equality holds in \refenum{ii} or \refenum{iii}, then $\bar D$ is a union of rational curves and has $-3\chi({\bar D})$ nodes.
\end{lem}
We remark that there exist examples of non-Gorenstein stable surfaces for which   the inequalities \refenum{ii}  and \refenum{iii} of Lemma  \ref{lem: chi(D)} fail.
\begin{proof}
 The curve $\bar D$ has nodes by the classification of lc pairs. Recall that Diagram \eqref{diagr: pushout} is a pushout diagram in the category of schemes. In particular, the points of $D$ correspond to equivalence classes of points on $\bar D^\nu$ with respect to the relation generated  by $x\sim y $ if $\bar \nu(x) =\bar\nu (y)$ or $\tau(x) = y$. Note that if an equivalence class contains the preimage of a node of $\bar D$ then either it contains no fixed point of $\tau$ and the image point is a degenerate cusp or it contains exactly two fixed points of $\tau$ and the image is a $\IZ/2\IZ$-quotient of a degenerate cusp. (Compare the discussion in \cite[Sect.~4.2]{liu-rollenske12} and \cite[\S 4]{ksb88}.) 
 
 Thus of the $2\bar \mu$ preimages of nodes of $\bar D$ in $\bar D^\nu$  exactly $2\mu_2$ are fixed by $\tau$ and there are exactly $\bar\mu +\mu_2$ points in $D^\nu$ that map to images of nodes in 
 $D$.
 By the normalisation sequences we have
 \begin{gather*}\chi({\bar D^\nu}) = \chi({\bar D}) + \bar \mu,\\
  \chi(D) = \chi({D^\nu})-((\bar \mu+\mu_2)-(\mu_1+\mu_2)) = \chi({D^\nu})+\mu_1-\bar \mu.
 \end{gather*}
Combining this with the Hurwitz formula for $\bar D^\nu\to D^\nu$,  which gives 
\[\chi({D^\nu}) = \frac 1 2\chi({\bar D^\nu})+\frac\rho 4, \]
we get
\[\chi(D)  = \frac 12\chi({\bar D^\nu})+\frac\rho 4+\mu_1-\bar \mu=
\frac  12\left( \chi({\bar D})-\bar \mu\right)+\frac\rho 4+\mu_1\]
 as claimed in \refenum{i}.

 Now assume in addition that $K_{\bar X}+\bar D$ is Cartier. Then, by  adjunction (see e.g. \cite[Sect.~4.1]{KollarSMMP}),
   $K_{\bar D}=(K_{\bar X}+\bar D)|_{\bar D}$ is ample, so $\bar D$ is a stable curve.
 Therefore, every rational component of the normalisation has at least three marked points mapping to nodes in $\bar D$ and thus $\chi({\bar D^\nu})\leq 2\bar \mu/3$ which implies $-\bar \mu\geq 3\chi({\bar D})$. This gives \refenum{ii} and proves the last sentence in the statement. 

 Equality in \refenum{ii} is attained if and only if  $\bar D^\nu$ consists of $-2\chi({\bar D})$  rational curves, each with three marked points; then the curve $\bar D$ has $-3\chi({\bar D})$ nodes.

In order to prove \refenum{iii}, we only need to show that if equality occurs in \refenum{ii} and $X$ is Gorenstein, then there is at least one degenerate cusp. But if equality holds in \refenum{ii}   then  $\bar D$ has $-3\chi(\bar D)>0$ nodes and, since   $X$ is Gorenstein,   each node of $\bar D$ maps to a degenerate cusp, that is,  $\mu_1>0$.
\end{proof}

\begin{thm}\label{thm: possible chi for normalisation types}
There exists a non-normal Gorenstein stable surface with normalisation of given type (as defined and classified in Section \ref{sec:pairs}) exactly in the following cases:
 \begin{center}
 \begin{tabular}{c ccc c}
 \toprule
\textrm{normalisation} & $\chi({ X})=0$& $\chi({ X})=1$& $\chi({ X})=2$& $\chi({ X})=3$\\
\midrule
$(P)$ 	& \checkmark 	& \checkmark	& \checkmark	& \checkmark\\
$(dP)$ 	&  	& \checkmark	& \checkmark	& \checkmark\\
$(E_-)$	&  	& 	& \checkmark	& \checkmark\\
$(E_+)$	&  	& \checkmark	& \checkmark	& \\
\bottomrule
\end{tabular}
\end{center}
\end{thm}
One could extend the above numerical analysis to all stable surfaces with $K_X^2=1$ and Gorenstein normalisation $(\bar X, \bar D)$. From a moduli perspective such surfaces do not form a good class:  they would include some but not all $2$-Gorenstein surfaces. 
\begin{proof}
 The restrictions follow from Proposition~\ref{prop: invariants}, 
 the invariants given in Table \ref{tab: invariants pairs} and Lemma \ref{lem: chi(D)} where in the cases $(E_\pm)$ we use that not all components of $\bar D$ can be rational. 
 
 The existence of examples is settled below in Section \ref{sect: examples}.
\end{proof}
The above results allow us to refine in the case $K^2=1$ the $P_2$-inequality $\chi\ge -K^2$,  proved in  \cite{liu-rollenske13} for Gorenstein stable surfaces: 
\begin{cor}\label{cor: chi>=0} If  $X$ is a Gorenstein stable surface with $K_X^2=1$, 
then    $\chi(X)\ge 0$.
\end{cor}
\begin{proof} Let $X$ be a Gorenstein stable surface with $K_X^2 =1$. 
If $X$ is normal then $\chi(X)\geq 1$ by \cite[Theorem~2]{bla94}. If $X$ is not normal then $\chi(X)\geq 0$ by Theorem \ref{thm: possible chi for normalisation types}.
\end{proof}

\subsection{Examples}\label{sect: examples}
For completeness, we now provide explicit examples for each case given in Theorem \ref{thm: possible chi for normalisation types}. We will analyse  such surfaces more systematically in  a subsequent paper.

 By Theorem \ref{thm: triple} and Corollary \ref{cor: main motivation} for each type we need to specify a (nodal) boundary $\bar D$ and an involution $\tau$ on the normalisation of $\bar D$ which induces a fixed point-free action on the preimages of the nodes. The holomorphic Euler-characteristic is then computed by Proposition \ref{prop: invariants}.

\begin{description}
 \item[ The case $(P)$] Examples  with $0\leq \chi(X)\leq 3$ are given in \cite[Sect.~5.1]{liu-rollenske13}.
\item[{The case $(dP)$}] ~
\begin{itemize}[leftmargin=*]
 \item Take $\bar D$ to be a general section in $|-2K_{\bar X}|$, which is smooth, and $\tau$ the hyperelliptic involution. This gives $\chi(X)=3$.

 \item Let $E_1, E_2\in |-K_{\bar X}|$ be two distinct smooth  isomorphic curves and fix the intersection point as a base point on both.  Let $\bar D = E_1+E_2$ and let $\tau$ be the involution that exchanges the two  curves preserving the base-point. Then $\chi(X)=2$.

\item Assume that $|-K_X|$ contains two distinct nodal plane cubics and let  $\bar D$ be their  union. The normalisation $\bar D^\nu$ consists of two copies of $\IP^1$ each with three marked points which are the preimages of the nodes of $\bar D$.

An involution on $\bar D^\nu$ interchanging the components is uniquely determined by its action on the marked points and we can choose it in such a way  that the preimage of the base-point of the pencil is not preserved by the involution (see Figure \ref{fig: construction}).
 One can easily see that this gives a rational curve of genus 2 (not nodal) as non-normal locus, thus $\chi(X)=1$.
{\scriptsize
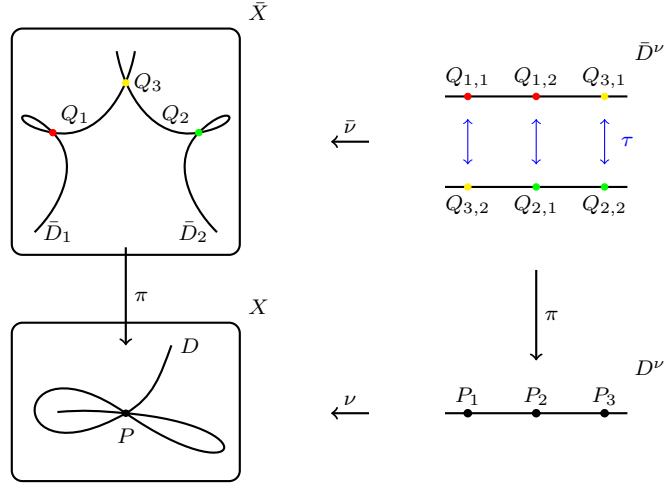
\begin{figure}\label{fig: construction}\caption{The construction of a numerical Godeaux surface with normalisation of type $(dP)$}
\begin{tikzpicture}
[curve/.style ={thick, every loop/.style={looseness=10, min distance=30}},
q1/.style={color=red},
q2/.style={color=green},
q3/.style={color=yellow},
scale = 0.6
]

\node (Q3) at (0,2.3){};
\node (Q2) at (1.6,1.2) {};
\node (Q1) at (-1.6,1.2) {};

\begin{scope}
\draw[thick, rounded corners] (-2.5, -1.5) rectangle (2.5, 3.5) node [above right] {$\bar X$};
\draw [curve] (-2,-1) node [ right] {$\bar D_1$} to[out=45, in=315] (-1.6, 1.2) to[out=135, in =170,loop] () to[out=350, in = 260] (.2,3);
\draw [curve, xscale=-1] (-2,-1)  node [left] {$\bar D_2$} to[out=45, in=315] (-1.6, 1.2) to[out=135, in =170,loop] () to[out=350, in = 260] (.2,3);

\filldraw[q3] (Q3) circle (2pt) ; \node at (Q3) [right] {$Q_3$};
\filldraw[q2] (Q2) circle (2pt);  \node at (Q2) [above left] {$Q_2$};
\filldraw[q1] (Q1) circle (2pt); \node at (Q1) [above right] {$Q_1$};

\node [corners={(-2.5, -1.5) (2.5,3.5)}] (Xbar) {};
\end{scope}

\begin{scope}[yshift=-5cm]
\draw[thick, rounded corners] (-2.5, -1.5) rectangle (2.5, 2) node [above right] {$X$};
\draw[curve,  every loop/.style={looseness=30, min distance=90}] (1,1.5) node [right]{$D$} to[out = 250, in = 30]
(0,0) to[out = 210, in = 145, loop]
() to[out = 325, in =355 ,  loop]
() to[out =175, in = 5] (179:1.5cm);

\filldraw (0,0) circle (2pt) node [below, yshift = -.1cm] {$P$};

\node [corners={(-2.5, -2) (2.5,2)}] (X) {};
\end{scope}

\begin{scope}[xshift=8cm, yshift=1cm]
\draw [curve] (3,1) -- (2.5,1) node[above] {$Q_{3,1}$} -- (1,1)  node[above] {$Q_{1,2}$} to (-0.5,1)  node[above] {$Q_{1,1}$}--(-1,1);
\filldraw [q3] (2.5,1) circle (2pt);
\filldraw[q1] (1,1) circle (2pt);
\filldraw[q1] (-0.5,1) circle (2pt);

\filldraw [curve] (3,-1)--(2.5,-1)  node[below] {$Q_{2,2}$}--(1,-1)  node[below] {$Q_{2,1}$}--(-0.5,-1)  node[below] {$Q_{3,2}$}--(-1,-1);
\filldraw [q2] (2.5,-1) circle (2pt);
\filldraw[q2] (1,-1) circle (2pt);
\filldraw[q3] (-0.5,-1) circle (2pt);

\node at (3.5, 2) {$\bar D^\nu$};

\begin{scope}[color= blue]
\draw[<->] (2.5, 0.5) -- (2.5, -0.5);
\draw[<->] (1, 0.5) -- (1, -0.5);
\draw[<->] (-.5, 0.5) -- (-.5, -0.5);
\node at (3,0) {$\tau$};
\end{scope}

\node [corners={(-1, -1.5) (3,1.5)}] (Dbarnu) {};
\end{scope}

\begin{scope}[xshift=8cm, yshift = -5cm]
\filldraw [curve] (3,0) --(2.5,0) circle (2pt) node[above] {$P_{3}$}--(1,0) circle (2pt) node[above] {$P_2$}--(-0.5,0) circle (2pt) node[above] {$P_1$}--(-1,0);

\node at (3.5, 1) {$ D^\nu$};

\node [corners={(-1, -.5) (3,.5)}] (Dnu) {};
\end{scope}

\begin{scope}
\draw [thick, ->]  (X.north) to node[right] {$\pi$} (Xbar.south);
\draw [thick, ->] (Dbarnu.west) to node[above] {$\bar\nu$} (Xbar.east);
\draw [thick, ->] (Dbarnu.south) to node[right] {$\pi$} (Dnu.north);
\draw [thick, ->] (Dnu.west) to node[above] {$\nu$} (X.east);
\end{scope}

\end{tikzpicture}
\end{figure}
}
\end{itemize}
\item[The case $(E_-)$] The divisor $\bar D$ is a curve of arithmetic genus 2, which after pullback to the minimal resolution becomes a degree 2 cover of the base curve of the projective bundle. If $\bar D$ is smooth,
 choosing as $\tau$ either the hyperelliptic involution or the involution corresponding to the double cover of the elliptic base curve gives the two possible values for $\chi(X)$.
\item[The case $(E_+)$]~
\begin{itemize}[leftmargin=*]
 \item A general $\bar D$ is a smooth curve of genus two and choosing $\tau$ to be the hyperelliptic involution we get $\chi(X) = 2$.
\item For the numerical Godeaux case let $E\isom \IC/\IZ[\I]$. Then multiplication by $1+\I$ induces an endomorphism of degree 2 on $E$, that is, an isomorphism $E\isom E/\xi$ for a particular 2-torsion element in $E$.
We can choose $\bar D \isom E\cup E/\xi \isom E\cup E$ in case $(E_+)$ (cf. \cite[\S 2]{catanese-ciliberto93})
and the intersection of the two components is a single point. Thus there is an involution $\tau$ on $\bar D$ with quotient $E$ which exchanges the two components while keeping the base-point.
With this choice $\chi(X)=1$.
\end{itemize}
\end{description}

\section{Normal Gorenstein stable surfaces with $K^2=1$}\label{section: bigandnef}

In this section we complement the results of Section \ref{sec:pairs}  by omitting the condition that the boundary should be non-empty, that is, we study Gorenstein log-canonical surfaces $X$ with $K_X$ ample and $K_X^2=1$. In the terminology of Section \ref{section: definitions} these are normal Gorenstein stable surfaces and they occur in the compactified Gieseker moduli space.

Of course, in this case  we cannot hope for a complete picture:  for instance surfaces of general type with $K^2=\chi=1$, known as Godeaux surfaces, have been an object of study for decades and a full classification has not been achieved yet.

Still, we are able to give a rough description according to  the Kodaira dimension of $\wt X$:\begin{thm}\label{thm:normal-case} Let $X$ be a normal Gorenstein stable  surface with $K^2_X=1$ and let $\epsi\colon \wt X\to X$ be its minimal desingularization. Then
\begin{enumerate}
\item If $\kappa(\wt X)=2$, then $X$ has canonical singularities.
\item If $\kappa(\wt X)=1$, then $\wt X$ is a minimal properly elliptic surface and $X$ has precisely one elliptic singularity of degree $1$.
\item If $\kappa(\wt X)=0$, denote by $X_{\min}$ the minimal model of $\wt X$. Then there exists  a nef effective  divisor  $D_{\min}$ on $X_{\min}$ and a point $P$ such that:
\begin{itemize}
\item $D^2_{\min}=2$ and $P\in D_{\min}$ has multiplicity $2$
\item  $\wt X$ is the blow-up of $X_{\min}$ at $P$
\item $X$ is obtained from $\wt X$ by blowing down the strict transform  of $D_{\min}$ and it has  either   one elliptic singularity of degree 2 or  two elliptic singularities of degree 1.
\end{itemize}
\item If $\kappa(\wt X)=-\infty$, then there are two possibilities:
\begin{itemize}
\item[(a)] $\chi(\wt X)=1$ and $\wt X$ has 1 or 2 elliptic singularities
\item[(b)] $\chi(\wt X)=0$, $\wt X$ has 1, 2 or 3  elliptic singularities; in this case, the exceptional divisors arising from the elliptic singularities are
smooth elliptic curves.
\end{itemize}
\end{enumerate}
\end{thm}
One can show that all cases actually occur (see for example \cite{fpr14a}).
The proof of Theorem \ref{thm:normal-case} occupies the rest of the section. We  fix set-up and notations to be kept throughout: $X$ is a normal Gorenstein stable  surface with $K^2_X=1$, $\epsi\colon \wt X\to X$ is the minimal resolution and $\wt L:=\epsi^*K_X$, so $\wt L$ is a nef and big line bundle with $\wt L^2=1$. One has $\wt L=K_{\wt X}+\wt D$, where $\wt D$ is effective and $\wt L \wt D=0$. It follows in particular that $\wt LK_{\wt X}=1$.

By the classification of normal Gorenstein lc singularities (cf. \cite[Thm. 4.21]{ksb88}), the singularities of $X$ are either canonical or elliptic.
The elliptic Gorenstein singularities are described in \cite[4.21]{reid97}: denoting by $x_1, \dots x_k\in X$ the elliptic singular points, we can  write $\wt D=\wt D_1+\dots +\wt D_k$, where  $\wt D_i$ is a divisor  supported on  $\epsi\inv (x_i)$ such that $p_a(Z)<p_a(\wt D_i)=1$ for every $0<Z<\wt D_i$. The divisors $\wt D_i$ are called the  {\em elliptic cycles} of $\wt X$.
 The degree  of the elliptic singularity $x_i$ is the positive  integer $-\wt D_i^2$.

The invariants of $X$ and $\wt X$ are related as follows:
\begin{lem}\label{lem:invariants-normal} In the above set-up:
$$p_g(X)=h^0(\wt L)\ge p_g(\wt X), \quad  q(X)\le q(\wt X) \quad \chi(X)=\chi(\wt X)+k.$$
\end{lem}
\begin{proof}
By the projection formula we  have $h^0(\wt L)=h^0(\epsi_*\wt L)=h^0(K_X)=p_g(X)$; in addition there is an inclusion $H^0(K_{\wt X})\hookrightarrow H^0(\wt L)$, since $\wt D$ is effective.

The remaining inequalities  follow by the 5-term exact sequence associated with the Leray spectral sequence for $\OO_{\wt X}$:
$$0\to H^1(\OO_X)\to H^1(\OO_{\wt X})\to H^0(R^1\epsi_*\OO_{\wt X})\to H^2(\OO_X)\to H^2(\OO_{\wt X})\to 0,$$
since $R^1\epsi_*\OO_{\wt X}$ has length 1 at each of the points $x_1, \dots, x_k$ and is zero elsewhere.
\end{proof}

We start by dealing with the case $\kappa(\wt X)>0$.
\begin{lem}\label{lem:kappa>0} If $\kappa(\wt X)>0$, then there are the following possibilities:
\begin{enumerate}
\item $X$ has canonical singularities
\item $\wt X$ is a minimal properly elliptic surface and $X$ has precisely one elliptic singularity of degree $1$.
\end{enumerate}
\end{lem}
\begin{proof} Let  $\eta\colon \wt X\to X_{\min}$  be the morphism to the  minimal model. Let   $M= \eta^*K_{X_{\min}}$, so that $K_{\wt X}=M+E$, where $E$ is exceptional for $\eta$.
We have $\wt L(M+E)=\wt L K_{\wt X}=\wt L ^2=1$. Since $\wt L$ is nef and big and some multiple of $M$ moves, we have $\wt L M=1$, $\wt L E=0$. Thus, since $\tilde L$ is the pullback of an ample divisor, $E$ is also contracted by $\epsilon$. Since $\epsilon$ is assumed minimal, there is no $\epsilon$-exceptional $(-1)$-curve,  while on the other hand $\eta$ is a composition of blow-ups of a smooth surface. Hence $E=0$, namely $\wt X$ is minimal.

If $\kappa(\wt X)=2$, then the index theorem  applied to $\wt L$ and $K_{\wt X}$ gives $K^2_{\wt X}=1$ and $K_{\wt X}$ and $\wt L$ are numerically equivalent (otherwise they span a 2-dimensional subspace on which the intersection form is positive). It follows that $\wt D\geq0$ is numerically trivial, hence $\wt D=0$  and $K_{\wt X}=\epsi^* K_X$, namely $X$ has canonical singularities.

If $\kappa(\wt X)=1$, then $\wt X$ is minimal properly elliptic and $K^2_{\wt  X}=0$. It follows that $(\wt D_1+\dots + \wt D_k)K_{\wt X}=\wt DK_{\wt X}=\wt L K_{\wt X}=1$.
Since  $\wt D_i K_{\wt X}>0$ for every $i$,  we have $k=1$, namely $\wt D$ is connected and $\wt D^2=-1$.
\end{proof}

Next we consider   the case $\kappa(\wt X)=0$:
\begin{lem}\label{lem:normal-kappa0}
 If $\kappa(\wt X)=0$, then 
$X$ is  as in Theorem \ref{thm:normal-case}, \refenum{iii}.
\end{lem}
\begin{proof}
 Let  $\eta\colon \wt X\to X_{\min}$  be the morphism to the  minimal model, so $\eta$ is a composition of $m$ blow-ups in smooth points $P_1,\dots P_m$, possibly infinitely near. Denote by $E_i$ the total transform on $\wt X$ of the exceptional curve that appears at the $i$-th blow-up:  then $E_i^2=E_iK_{\wt X}=-1$, $E_iE_j=0$ if $i\ne j$, and $K_{\wt X}$ is numerically equivalent to $\sum _{i=1}^mE_i$. Observe that each $E_i$ contains at least one irreducible $(-1)$-curve.
Since  $\epsilon$ is relatively minimal,  $\tilde L$ is  positive on  irreducible 
$(-1)$-curves. Hence we have $1=\wt L K_{\wt X}=\sum_{i=1}^m\wt LE_i\ge m$, and we conclude that $m=1$, i.e., $\epsi$ is a single blow-up. We set $E=E_1$.
 
Write $\wt D=\wt D_1+\dots +\wt D_k$, with the $\wt D_i$ disjoint elliptic cycles.
We have $2=(\tilde L - K_{\tilde X})K_{\tilde X}=\wt DK_{\wt X}=\wt D_1K_{\wt X}+\dots +\wt D_kK_{\wt X}$, thus either $k=1$ and $2=\wt D_1K_{\wt X}=\wt  D_1 E$, or $k=2$ and $1=\wt D_iK_{\wt X}=\wt D_iE$, for $i=1,2$.
In the former case we have $\wt D_1^2=\wt D^2=-2$, and  in the latter  case we have $\wt D_1^2=\wt D_2^2=-1$, since $p_a(\wt D_i)=1$.

 We set $D_{\min}=\eta_*\wt D$. The divisor $D_{\min}$ has $D_{\min}^2=2$ and contains $P$ with multiplicity 2.

In order to complete the proof we need to show that $D_{\min}$ is nef. Let  $\Gamma$ be   an irreducible curve of $X_{\min}$  and write $\eta^*\Gamma=\wt \Gamma+\alpha E$, where $\wt \Gamma$ is the strict transform and $\alpha\ge 0$. We have $\Gamma D_{\min}=(\epsi^*\Gamma)(\epsi^*D_{\min})=\epsi^*\Gamma(\wt L+E)=\epsi^*\Gamma \wt L\ge 0$, since $\wt L$ is nef. 
\end{proof}

Finally we consider the case $\kappa(\wt X)=-\infty$:
\begin{lem} If  $\kappa(\wt X)=-\infty$,  then there are the following possibilities:
\begin{itemize}
\item[(a)] $\chi(\wt X)=1$ and $\wt X$ has 1 or 2 elliptic singularities
\item[(b)] $\chi(\wt X)=0$, $\wt X$ has 2 or 3  elliptic singularities and
$\wt D$ is a union of disjoint smooth elliptic curves.
\end{itemize}
\end{lem}

\begin{proof}
Since $\wt X$ is ruled, we have $\chi(\wt X)\le 1$, with equality if and only if $\wt X$ is rational. 

Assume $\chi(\wt X)\le 0$ and let $a\colon X\to B$ be the Albanese map, where $B$ is a smooth curve of genus $b>0$. Write $\wt D=\wt D_1+\cdots +\wt D_k$; since the general  fiber of $a$ is a smooth rational curve and $p_a(\wt D_i)=1$ for all $i$,  no $\wt D_i$ can be contracted  to a point  by $a$, hence $\wt D_i$ dominates $B$. It follows that $b=1$ and $\wt D_i$ contains a smooth elliptic curve $D'_i$. Since $\wt D_i$ is minimal among the divisors $Z>0$  supported on $\epsi\inv(x_i)$ and such that  $p_a(Z)=1$, it follows that $\wt D_i=D'_i$.

One has $\chi(X)\ge 1$  by \cite[Theorem~2]{bla94} and $\chi(X)\le 3$ by the stable Noether inequality for normal Gorenstein stable surfaces \cite{sakai80, liu-rollenske13}. Since $k>0$, Lemma  \ref{lem:invariants-normal} gives $1\le k\le  3$ if $\chi(\wt X)=0$ and $1\le k\le  2$ if $\chi(\wt X)=1$.\end{proof}

\end{document}